\documentclass[border=1mm]{amsart}
\usepackage[utf8]{inputenc}
\usepackage{amsmath}
\usepackage{lipsum}
\usepackage{graphicx}
\usepackage{amsmath}
\usepackage{amsmath,kbordermatrix,blkarray}
\renewcommand{\kbldelim}{(}
\renewcommand{\kbrdelim}{)}
\usepackage{csquotes}
\usepackage{placeins}
\usepackage{tikz}
\usepackage{stackengine}
\usepackage{amssymb}
\usepackage{mathtools}
\usepackage{blkarray, bigstrut}
\usepackage{amsthm}
\usepackage{enumitem}
\usepackage{amsthm}
\usepackage{upgreek}
\usepackage[english]{babel}
\usepackage{wasysym}
\usepackage{blkarray}
\usepackage{textcomp}
\usepackage{pifont}
\usepackage{enumitem} 
\usepackage{caption}
\usepackage{url}
\usepackage{float}
\usepackage{xcolor}
\usepackage{hyperref}
\usetikzlibrary{arrows.meta, calc, positioning}

\tikzset{
error/.style = {circle, draw, fill=red, align=center,
                inner sep=0pt},
block/.style = {rectangle, draw, fill=blue!20, rounded corners,
                text width=5em, minimum height=4em, align=center},
        }
\usepackage{mathtools}

\usepackage{ textcomp }

\newtheorem{theorem}{Theorem}[section]
\newtheorem{Lemma}[theorem]{Lemma}

\title[Fusions of the Tensor Square of a Strongly Regular Graph]{Fusions of the Tensor Square of a Strongly Regular Graph}

{\author[A.~Herman]{Allen Herman\textsuperscript{1,*}}}
\thanks{\textsuperscript{*} The work of this author has been supported by NSERC.}
\address{\textsuperscript{1} Department of Mathematics and Statistics, University of Regina, Regina, Saskatchewan S4S 0A2, Canada}
\email{Allen.Herman@uregina.ca}
{\author[N.~Joshi]{Neha Joshi\textsuperscript{2}}}
\address{\textsuperscript{2} Department of Mathematics and Statistics, University of Regina, Regina, Saskatchewan S4S 0A2, Canada}
\email{Njp008@uregina.ca}

\begin{document}

\begin{abstract}
In this paper we determine all fusions of the association scheme $\mathcal{A} \otimes \mathcal{A}$, where $\mathcal{A}$ is the symmetric rank $3$ association scheme corresponding to a strongly regular graph.  This includes both guaranteed fusions, which are fusions for all symmetric rank $3$ association schemes $\mathcal{A}$, and specific case fusions, which only exist under restrictions on the parameters of the association scheme.  Along the way we will determine the fusions of wreath products of strongly regular graphs and the fusions of the tensor square of a symmetric rank $3$ table algebra.  This extends recent work of the authors and Meagher, which solved the same problem for the generalized Hamming scheme $H(2,\mathcal{A})$ of the association scheme obtained from a strongly regular graph.  The main results of this article show (1) the families of strongly regular graphs for which $\mathcal{A} \otimes \mathcal{A}$ has a special case fusion are the same families for which $H(2,\mathcal{A})$ has a special case fusion; and (2) the imprimitive strongly regular graphs are the only family of strongly regular graphs for which the wreath product $\mathcal{A} \wr \mathcal{A}$ has a special case fusion.  
\end{abstract}



\maketitle

\section{Introduction}

In \cite{HermanJoshiMeagher2022}, the authors and Karen Meagher determined all fusions of the generalized Hamming scheme $H(2,\mathcal{A})$, where $\mathcal{A}$ is the symmetric rank $3$ association scheme corresponding to a strongly-regular graph.  In contrast with earlier work of \cite{FadRussian} (see \cite{FadRussianTranslated}) that worked directly with the intersection numbers, the approach in \cite{HermanJoshiMeagher2022} was to use the the Bannai-Muzychuk criteria on the character table to establish the existence of each fusion.  As pointed out in \cite{HermanJoshiMeagher2022}, the adjacency algebra of $H(2,\mathcal{A})$ is isomorphic to the rank $6$ symmetric $2$-tensor subalgebra $Sym^2(\mathcal{A})$ of $\mathcal{A} \otimes \mathcal{A}$, so the authors asked if the fusions of other subalgebras of $\mathcal{A} \otimes \mathcal{A}$ could be classified in a similar fashion, most importantly for the wreath product $\mathcal{A} \wr \mathcal{A}$ and for the full tensor product $\mathcal{A} \otimes \mathcal{A}$.  The main results of this article classify these fusions.  As the wreath product has rank $5$, the problem is a bit easier than for the full tensor product, and does not require the assistance of a computer. For the tensor product, each of the $4140$ partitions of the set of $8$ nonidentity elements must be considered, as compared to the set of $52$ partitions that needed to be considered in \cite{HermanJoshiMeagher2022}.  To manage the extra partitions, the authors have implemented a sieve that eliminates the partitions whose corresponding fusion would not be consistent with the parameters of a symmetric rank $3$ table algebra.  For those that are consistent, the Bannai-Muzychuk criterion is applied to verify the partition gives a fusion under the conditions on parameters it imposes.  

Our technique applies more generally in the setting where $\mathcal{A}$ is the standard basis of a symmetric rank $3$ table algebra, so the main results can be interpreted in that setting.  For those unfamiliar with table algebras, the connection between tensor products of association schemes and table algebras is explained by Xu in \cite{Xu2013}.  That the wreath product of two association schemes is a fusion of their tensor product was noted by Song in \cite{Song2002}. Our results on rank $3$ fusions overlap earlier results of Sankey \cite{Sankey}.  
 
\section{Preliminaries}

\subsection{Parameters of strongly regular graphs}

Let $\Gamma$ be a strongly regular graph on a set of $n$ vertices.  One of the many equivalent definitions of a strongly regular graph is that the set of $n \times n$ $01$-matrices $\mathcal{A} = \{ A_0, A_1, A_2 \}$ is a basis for the adjacency algebra of an association scheme, where $A_0 =$ the $n \times n$ identity matrix, $A_1 =$ the adjacency matrix of $\Gamma$, and $A_2 =$ the adjacency matrix of the complement $\bar{\Gamma}$of $\Gamma$.  In particular, if $\Gamma$ is a strongly regular graph with parameters $(n,k,\mu,\nu)$, then these matrices satisfy $$A_1^2 = k A_0 + \mu A_1 + \nu A_2, ~~  A_1 A_2 = (k-1-\mu)A_1 + (k - \nu) A_2,$$ and $$A_2^2 = (n-1-k) A_0 + (n-2k+\mu)A_1 + (n-2k+\nu-2)A_2.$$  The algebra spanned by $\mathcal{A} = \{ A_0, A_1, A_2\}$ is thus exactly isomorphic to a symmetric rank $3$ standard integral table algebra spanned by the set $\mathcal{B} = \{ b_0, b_1, b_2 \}$ of left regular matrices $b_0 = I_3$, 
$$ b_1 = \begin{bmatrix} 0 & k & 0 \\ 1 & \mu & k-1-\mu \\ 0 & \nu & k - \nu \end{bmatrix}, \mbox{ and } b_2 = \begin{bmatrix} 0 & 0 & n-k-1 \\ 0 & k-1-\mu & n-2k+\mu \\ 1 & k-\nu & n-2k-2+\nu \end{bmatrix} $$
via the linear extension of the map $A_i \mapsto b_i$ for $i=0,1,2$.
Since this is an integral table algebra, the entries of $b_1$ and $b_2$ are nonnegative real numbers and the valencies $k$ and $\ell:=n-k-1$ are positive integers.  It is also a standard integral table algebra because the coefficients of the identity $b_0$ when $b_1^2$ and $b_2^2$ are expressed in terms of the basis $\mathcal{B}$ are the respective valencies $k$ and $\ell$.  Furthermore, the eigenvalues of $A_1$ and $A_2$ give the entries of the {\it character table} (a.k.a.~{\it first eigenmatrix}) of the association scheme $\mathcal{A}$: 
\[
\mathcal{P}(\mathcal{A}) = 
    \begin{blockarray}{cccc}
      \begin{block}{[ccc]c}
       1 & k & \ell & 1\\
       1 & r & -1-r  & f\\
    1 & s & -1-s  & g\\
      \end{block}
    \end{blockarray}.
\]
Each row of $\mathcal{P}$ corresponds to one eigenspace of $A_1$, with the entries in that row being the eigenvalues taken by $A_0$, $A_1$, and $A_2$, respectively, on that eigenspace. The number on the right is the dimension of this eigenspace, which gives the {\it multiplicity} of the eigenvalue of $A_1$.  The first common eigenspace for $A_1$ and $A_2$ is spanned by the all $1$'s vector, so it has multiplicity $1$ and the eigenvalues of $A_1$ and $A_2$ on its eigenvectors are the valencies $k$ and $\ell$, respectively.  The row orthogonality relations tell us the last two rows sum to $0$, 
that 
\begin{equation*}
   1 + \frac{rs}{k} + \frac{(-1-r)(-1-s)}{\ell} = 0, 
\end{equation*}  
and also give us formulas for the multiplicities $f$ and $g$: 
$$ 1 + \frac{r^2}{k} + \frac{(-1-r)^2}{\ell} = \frac{n}{f}, \mbox{ and } 
1 + \frac{s^2}{k} + \frac{(-1-s)^2}{\ell} = \frac{n}{g}. $$ 
The column orthogonality relation on the first two columns tells us  
$$ k + fr + gs = 0.$$

\begin{Lemma}\label{parameters}
Assume the last two rows of the character table are ordered so that $r \ge s$.  Then the character table values $k$, $\ell (=n-k-1)$, $r$, and $s$ satisfy 
\begin{enumerate}
    \item $\ell=\frac{-k(1+r+s+rs)}{(k+rs)}$;
    \item $k, \ell \ge 1$ and $k \ge r \ge 0 > -1 \ge s = \frac{-(k+kr+k\ell)}{(k+kr+r\ell)}$; 
    \item $\ell \ge -1-s \ge 0 > -1 \ge -1-r$;
    \item $(k+rs) \ge 0 \ge (1+r+s+rs)$; and 
    \item $\ell + (1 + r + s + rs) \ge 0$ and $\ell -1 + rs \ge 0$.  
\end{enumerate}
\end{Lemma} 

\begin{proof} 
First, the identities $\ell = \frac{-k(1+r+s+rs)}{(k+rs)}$ and $s = \frac{-(k+kr+k\ell)}{(k+kr+r\ell)}$ are consequences of the row orthogonality relation applied to the last two rows of $\mathcal{P}$. 

For the other properties, we use the fact that the columns of $\mathcal{P}$ form the basis of an algebra under entrywise multiplication that has the same structure constants as the basis $\mathcal{B} = \{b_0, b_1, b_2 \}$ (which also matches that of the basis $\mathcal{A}$).  This allows us to determine that the regular matrix of $b_1$ in the basis $\mathcal{B}$ to be 
$$\begin{bmatrix} 0 & k & 0 \\ 1 & k + r + s + rs & -(1 + r + s + rs) \\ 0 & k + rs & -rs \end{bmatrix}.$$   
Since $\mathcal{B}$ is the basis of a table algebra, the entries of this regular matrix are nonnegative real, which implies $1+r+s+rs \le 0$ and $k+rs \ge 0$.  Also, $k$ is the Perron-Frobenius eigenvalue of $A_1$, so $k \ge r,s \ge -|k|$.  Since $r > s$, $-rs \ge 0$ implies $r \ge 0 > s$ or $r > 0 \ge s$.  But $r \ge 0$ and $1 + r + s + rs \le 0$, so we must have $s < 0$.  Since $k+r+s+rs \ge 0 \ge 1+r+s+rs$, we must have $k \ge 1$.  Since we can interchange columns of $\mathcal{P}$ at the start we can also conclude $\ell \ge 1$. 
From the formula for $s = \frac{-(k+kr+k\ell)}{(k+kr+r\ell)}$, we can now see that having $0 > s > -1$ would imply that $r > k$, which is false, so $-1 \ge s$. 

\vspace{0.2cm}

We can also find the regular matrix of $b_2$ with respect to the basis $\mathcal{B}$, it is 
$$ \begin{bmatrix} 0 & 0 & \ell \\ 0 & -(1+r+s+rs) & \ell + (1 + r + s + rs) \\ 1 & -rs & \ell - 1 + rs \end{bmatrix}.$$
Since this is a nonnegative matrix when $\mathcal{B}$ is a table algebra, we can conclude $\ell + (1 + r + s + rs), \ell - 1 + rs \ge 0$. 
\end{proof}

\vspace{0.3cm}

The case $k=r$ (which is equivalent to that of $s=-1$) occurs in the case when the strongly-regular graph $\Gamma$ is disconnected and its complement is a complete multipartite graph.  When the roles of $A_1$ and $A_2$ are interchanged, this is the case $\ell=-1-s$ and $r=0$.  We will refer to these cases as being {\it imprimitive}, and the case $k>r$ and $-1>s$ as being primitive.  In the imprimitive case, if the graph $\Gamma$ is the disjoint union of $m+1$ copies of the complete graph on $k+1$ vertices, then the character table of the association scheme is 
\[
\mathcal{P} = 
    \begin{blockarray}{cccc}
      \begin{block}{[ccc]c}
     1 & k & m(1+k)  & 1\\
       1 & k & -1-k & m\\
    1 & -1 & 0  & k(1+m)\\
      \end{block}
    \end{blockarray}.
\]
As we will see, the imprimitive case will be the one where $\mathcal{A} \otimes \mathcal{A}$ has the most special case fusions. 

\subsection{The tensor square association scheme}\label{notation}

The \textit{tensor product scheme}, or \textit{Kronecker product}, of any two association schemes with sets of adjacency matrices $\mathcal{A} = \{A_0=I_n,A_1,\dots,A_d\}$ and $\mathcal{B}=\{B_0=I_{n'},B_1,\dots,B_{d'}\}$ is the association scheme of order $nn'$ and rank $(d+1)(d'+1)$ whose adjacency matrices are $\mathcal{A} \otimes \mathcal{B} = \{A_i \otimes B_j : 0 \le i \le d, 0 \le j \le d'\}$.  We will be interested in the tensor square $\mathcal{A} \otimes \mathcal{A}$ when $\mathcal{A}=\{A_0,A_1,A_2\}$, so this association scheme has rank $9$.  

We will use two convenient shorthand notations for elements of $\mathcal{A} \otimes \mathcal{A}$.  For $i,j \in \{0,1,2\}$, we write $A_{ij}$ for the element $A_i \otimes A_j$.  Later, we will make use of a single-index notation, which is convenient for identifying partitions of $\mathcal{A} \otimes \mathcal{A} - \{A_{00}\}$.  If we identify $A_{ij}$ with $C_{3j+i+1}$, for $i,j \in \{0,1,2\}$, then we can identify partitions of $\mathcal{A} \otimes \mathcal{A} - \{A_{00}\}$ with partitions of $\{2,\dots,9\}$, which we will also write with a shorthand notation.  For example, if $\mathcal{T}_{\mathcal{A}} = \{ A_0, A_1 + A_2 \}$ is the basis of the trivial fusion of $\mathcal{A}$, then the tensor product subalgebras with basis $\mathcal{T}_{\mathcal{A}} \otimes \mathcal{T}_{\mathcal{A}}$, $\mathcal{T}_{\mathcal{A}} \otimes {\mathcal{A}} $, and ${\mathcal{A}} \otimes \mathcal{T}_{\mathcal{A}}$ can be identified with the partitions $23|47|5689$, $23|4|56|7|89$, and $2|3|47|58|69$.  

What makes such a collection correspond to a {\it fusion} of $\mathcal{A} \otimes \mathcal{A}$ is that this subset of elements is a basis for the unital subalgebra of $\mathbb{C}[\mathcal{A} \otimes \mathcal{A}]$ that it generates.   This is a special situation, as in most cases a collection of disjoint sums of basis elements will generate a subalgebra of larger dimension than the size of the collection. 

$\mathcal{A} \otimes \mathcal{A}$ has three types of fusion subalgebras that are guaranteed to exist for all rank $3$ table algebra bases $A$.  The first type is the tensor product subalgebra type, which includes $\mathcal{A} \otimes \mathcal{A}$ and the three tensor products involving $\mathcal{T}_{\mathcal{A}}$ previously introduced.  The second is the symmetric tensor square subalgebra $Sym^2(\mathcal{A})$, which is known as the generalized Hamming scheme $H(2,\mathcal{A})$ when $\mathcal{A}$ corresponds to a strongly-regular graph.  Its basis is equal to the set of the elementary symmetric $2$-tensors
$$Sym^2(\mathcal{A}) = \{ A_{00}, A_{10}+A_{01}, A_{20}+A_{02}, A_{11}, A_{21}+A_{12}, A_{22} \}.$$
 Using the indices obtained from $C_{3j+i+1}=A_{ij}$, we identify this fusion of $\mathcal{A} \otimes \mathcal{A}$ with the partition $24|37|5|68|9$.  The third type consists of wreath products.  The full wreath product $\mathcal{A} \wr \mathcal{A}$ occurs in two different ways, as 
$$ (\mathcal{A} \otimes 1) \wr (1 \otimes \mathcal{A}) = \{ A_{00}, A_{10}, A_{20}, A_{01}+A_{11}+A_{21}, A_{02}+A_{12}+A_{22} \} $$
and 
$$ (1 \otimes \mathcal{A}) \wr (\mathcal{A} \otimes 1) = \{ A_{00}, A_{01}, A_{02}, A_{10}+A_{11}+A_{12}, A_{20}+A_{21}+A_{22} \}. $$
These are associated with the partitions $2|3|456|789$ and $258|369|4|7$, respectively.
These wreath products have several proper wreath product fusions, all of which involve $\mathcal{T}_{\mathcal{A}}$: $(\mathcal{T}_{\mathcal{A}} \otimes 1) \wr (1 \otimes \mathcal{T}_{\mathcal{A}})$, $(\mathcal{T}_{\mathcal{A}} \otimes 1) \wr (1 \otimes \mathcal{A})$, $(\mathcal{A} \otimes 1) \wr (1 \otimes \mathcal{T}_{\mathcal{A}})$, $(1 \otimes \mathcal{T}_{\mathcal{A}}) \wr (\mathcal{T}_{\mathcal{A}} \otimes 1)$, $(1 \otimes \mathcal{T}_{\mathcal{A}}) \wr (\mathcal{A} \otimes 1)$, and $(1 \otimes \mathcal{A}) \wr (\mathcal{T}_{\mathcal{A}} \otimes 1)$.  These can be identified, in order, with the six partitions $$23|456789, ~23|456|789, ~2|3|456789, ~235689|47, ~258|369|47,~ \mbox{and} ~235689|4|7.$$  In the third section we will verify that the 13 fusions of $\mathcal{A} \otimes \mathcal{A}$ listed so far are the only fusions that are guaranteed, as any other fusion imposes conditions on the parameters on the symmetric rank $3$ table algebra. 

\subsection{The Bannai-Muzychuk criterion}  Given a basis $\mathcal{A}=\{A_0, A_1, \dots, A_d\}$ of a table algebra (or adjacency algebra) of rank $d$, and a partition $\uptau=\{T_0 = \{0\}, T_1, \dots, T_{d'} \}$ of $\{0,1,\dots,d\}$, we let $A_{T_j} = \sum_{t \in T_j} A_t$ for all $j \in \{0,1,\dots,d'\}$.  The partition $\uptau$ corresponds to a fusion of the table algebra when $A^{\uptau} = \{A_{T_j}\}$ is also the basis of a table algebra. 

Let $\mathcal{P}$ be the $(d+1)\times (d+1)$-character table of the table algebra corresponding to $\mathcal{A}$.  Given a partition $\uptau$ of $\{0,1,\dots,d\}$, we define $\mathcal{P}_{\uptau}$ to be the order $|\uptau|\times (d+1)$ matrix with the rows indexed by the classes $T \in \uptau$. The row corresponding to a class $T\in \uptau$ is the sum of rows in $\mathcal{P}$ indexed by $t$ for all $t\in T$. Since $\mathcal{P}$ is a non-singular matrix, the rank of $\mathcal{P}_{\uptau}$ is equal to $|\uptau|$. Hence, the number of distinct columns in $\mathcal{P}_{\uptau}$ is at least $|\uptau|$. The next Lemma, known as the \textit{Bannai-Muzychuk criterion} (see Muzychuk \cite{JohnsonMuzychuk}), states a necessary and sufficient condition for a partition $\uptau$ of the rows of $\mathcal{P}$ to produce a fusion.

\begin{Lemma}{\label{Bannai-Muzychuk}}
A partition $\uptau=\{T_{0}=\{0\},\dots,T_{d'}\}$ of the index set $\{0,\dots,d\}$ determines a fusion subalgebra with basis $\mathcal{A}^{\uptau}=\{A_0, A_{T_1}, \dots, A_{T_{d'}} \}$ if and only if the number of different columns in $\mathcal{P}_{\uptau}$ equals $|\uptau|$.
\end{Lemma}

\section{Fusions of the wreath product} 

If $\mathcal{P}(\mathcal{A})$ is the character table of $\mathcal{A}$, the character table of $\mathcal{A} \otimes \mathcal{A}$ is the Kr\"{o}necker product 
$\mathcal{P} \circ \mathcal{P}$.  When $\mathcal{A}$ is the basis of adjacency matrices of the association scheme corresponding to a strongly regular graph and the elements of $\mathcal{A} \otimes \mathcal{A}$ are labelled as in the previous section, the wreath product $(\mathcal{A} \otimes 1) \wr (1 \otimes \mathcal{A})$ corresponds to the partition $2|3|456|789$ of $\{2,\dots,9\}$.  When we apply the Bannai-Muzychuk criterion for this partition to the $9 \times 9$ character table $P \circ P$, we do indeed get exactly $5$ distinct rows, and these rows give us the character table of $(\mathcal{A} \otimes 1) \wr (1 \otimes \mathcal{A})$: 

\[
    \begin{blockarray}{ccccccc}
         & A_{00} & A_{10} & A_{20} & A_{01}+A_{11}+A_{21} & A_{02}+A_{12}+A_{22} & \\
      \begin{block}{c[ccccc]c}
 \chi_{00} & 1 & k & \ell & kn & \ell n & 1 \\
\chi_{01} & 1 & k & \ell & rn & n(-1-r) & m_r \\
\chi_{02} & 1 & k & \ell & sn & n(-1-s) & m_s \\
\chi_{11} & 1 & r & -1-r & 0 & 0 & nm_r \\
\chi_{21} & 1 & s & -1-s & 0 & 0 & nm_s \\
      \end{block}
    \end{blockarray}
\]

Our strategy to find all fusions of $\mathcal{A} \wr \mathcal{A}$ (with the above orientation) is to then apply the Bannai-Muzychuk criterion again to the $15$ partitions of $2|3|456|789$, and use it to determine the conditions on parameters for a partition to correspond to a special case fusion.  (As there are so few partitions, it is straightforward to work through all the possibilities by hand.  Nevertheless we will provide the details here to illustrate what our computer implementation accomplishes for the full tensor product in the next section.)  
  
\begin{theorem} Let $\mathcal{A}$ be the basis of a symmetric rank $3$ table algebra (i.e.~the set of adjacency matrices for the association scheme corresponding to a strongly-regular graph).  Consider the wreath product $(\mathcal{A} \otimes 1) \wr (1 \otimes \mathcal{A})$ as the fusion of $\mathcal{A} \otimes \mathcal{A}$ corresponding to the partition $2|3|456|789$ as described in \S \ref{notation}.  Then 

\begin{enumerate} 
\item $(\mathcal{A} \otimes 1) \wr (1 \otimes \mathcal{A})$ has guaranteed fusions corresponding to the partitions $2|3|456789$, $23|456|789$, and $23|456789$.  These correspond to wreath product subalgebras of $(\mathcal{A} \otimes 1) \wr (1 \otimes \mathcal{A})$ where either or both of the $\mathcal{A}$
s are replaced by the trivial fusion $\mathcal{T}_{\mathcal{A}}$.  Furthermore, these are the only nontrivial guaranteed fusions of $(\mathcal{A} \otimes 1) \wr (1 \otimes \mathcal{A})$.

\item $(1 \otimes \mathcal{A}) \wr (\mathcal{A} \otimes 1)$ has guaranteed fusions corresponding to the partitions $258|369|47$, $235689|4|7$, and $235689|47$, which correspond to wreath product subalgebras of $(1 \otimes \mathcal{A}) \wr (\mathcal{A} \otimes 1)$ where either or both of the $\mathcal{A}$'s are replaced by the trivial fusion $\mathcal{T}_{\mathcal{A}}$.  Furthermore, these are the only nontrivial guaranteed fusions of $(1 \otimes \mathcal{A}) \wr (\mathcal{A} \otimes 1)$.

\item The only special case fusions of $(\mathcal{A} \otimes 1) \wr (1 \otimes \mathcal{A})$ occur when $\mathcal{A}$ corresponds to an imprimitive strongly regular graph. The special case fusions in the case $k=r$ and $s=-1$ correspond to the partitions $2|3456789$, $23456|789$, and $2|3456|789$.  The special case fusions in the case $\ell=-1-s$ and $r=0$ correspond to the partitions $2456789|3$, $23789|456$, and $2789|3|456$.  

\item The only special case fusions of $(1 \otimes \mathcal{A}) \wr (\mathcal{A} \otimes 1)$ occur when $\mathcal{A}$ corresponds to an imprimitive strongly regular graph.  The special case fusions in the case $k=r$ and $s=-1$ correspond to the partitions $2345689|7$, $258|34679$, and $258|3469|7$.  The special case fusions in the case $\ell=-1-s$ and $r=0$ correspond to the partitions $2356789|4$, $24578|369$, and $2578|369|4$. 
\end{enumerate} 
\end{theorem} 

\begin{proof} 
First, note that since each of the three wreath products does give a table algebra, these do correspond to fusions of $(\mathcal{A} \otimes 1) \wr (1 \otimes \mathcal{A})$ that are guaranteed to exist for all symmetric rank $3$ table algebras $\mathcal{A}$.   

Now suppose $\uptau'$ is a partition other than these three for which $(\mathcal{A} \otimes 1) \wr (1 \otimes \mathcal{A})$ has a nontrivial fusion corresponding to $\uptau'$ for some symmetric rank $3$ table algebra $\mathcal{A}$.  There are ten possibilities for $\uptau'$, which we will consider one at a time.  

\vspace{0.2cm}

$\uptau'= 2|3456789$: Since $\chi_{21}(A_{10})=s$ is the only negative character value in its column, $\chi_{21}$ will be isolated.  By the Bannai-Muzychuk criterion, we must have $\chi_{01}=\chi_{02}=\chi_{11}$.  This implies $k=r$, which implies $s=-1$.  For this to be a rank $3$ fusion, we also need $\chi_{00} \ne \chi_{01}$.  Since $\ell+n(n-1) > \ell - n$, this is the case.  So this gives a fusion only in the imprimitive case $k=r$ and $s=-1$. 

\vspace{0.2cm}

$\uptau' = 3|2456789$: In this case $\chi_{11}$ will be isolated, so we must have $\chi_{01}=\chi_{02}=\chi_{21}$. Equating the values on $A_{20}$ we get $\ell = -1-s$, which implies $r=0$. So this will be a fusion in the other imprimitive case when $\ell=-1-s$ and $r=0$. 

\vspace{0.2cm}

$\uptau' = 23789|456$: In this case $\chi_{02}$ will be isolated, and we will have $\chi_{11}=\chi_{21}$.  Since the values of $\chi_{00}$ and $\chi_{01}$ on $A_{10}+A_{20}+A_{02}+A_{12}+A_{22}$ are $k + \ell + n\ell > k + \ell + n(-1-r)$, we can only have a rank $3$ fusion when $\chi_{01}=\chi_{11}=\chi_{21}$, which implies $rn=0$ and $k + \ell + n(-1-r)=-1$.  This implies $r=0$ and hence $\ell=-1-s$.  So this is a fusion in this imprimitive case. 

\vspace{0.2cm}

$\uptau' = 23456|789$: This time $\chi_{01}$ will be isolated and we will have $\chi_{11}=\chi_{21}$.  Since the values of $\chi_{00}$ and $\chi_{02}$ on $A_{10}+A_{20}+A_{01}+A_{11}+A_{21}$ satisfy $k+\ell+kn>k+\ell+sn$, the only way we can have a rank $3$ fusion is for $\chi_{02}=\chi_{11}$.  This requires $k+\ell+sn=-1$ and $n(-1-s)=0$, which holds when $s=-1$ and $k=r$.  So this gives a fusion when $k=r$ and $s=-1$.  

\vspace{0.2cm}

$\uptau'=2|3456|789$: Again we will have $\chi_{21}$ isolated.  Similarly $\chi_{01}$ takes only one negative value, $n(-1-r)$, on $A_{02}+A_{12}+A_{22}$, so it is also isolated.  So we must have that $\chi_{02}=\chi_{11}$, and that these are not equal to $\chi_{00}$.  Comparing the values of these characters on $A_{01}$ gives $k=r$, and on $A_{02}+A_{12}+A_{22}$ gives $n(-1-s)=0$.  So $k=r$ and $s=-1$.  Under these conditions, $\chi_{11}$ and $\chi_{00}$ agree on $A_{10}$ and $A_{02}+A_{12}+A_{22}$, but they take values $\ell+kn$ and $\ell-n$ on $A_{20}+A_{01}+A_{11}+A_{21}$, which are not equal.  So this does give a fusion in the case $k=r$ and $s=-1$.   

\vspace{0.2cm}

$\uptau'=2|3789|456$:  Again $\chi_{21}$ and $\chi_{02}$ must be isolated.  This implies the values of $\chi_{01}$ and $\chi_{11}$ will agree.  This gives $k=r$ and $kr=0$.  But then $r=0=k$, and this is a contradiction.  

\vspace{0.2cm}

$\uptau'= 2456|3|789$: In this case, $\chi_{01}$ and $\chi_{11}$ will be isolated, so in order for this partition to produce a fusion we must have $\chi_{02}=\chi_{21}$.  Equating their values on $A_{20}$ gives $\ell=-1-s$, and equating their values on $A_{02}+A_{12}+A_{22}$ we get $n(-1-s)=0$.  But this would imply $\ell = 0$, a contradiction. 

\vspace{0.2cm}

$\uptau'= 2789|3|456$: This time $\chi_{02}$ and $\chi_{11}$ will be isolated, so we must have $\chi_{01}=\chi_{21}$.  Equating their values on $A_{20}$ we get $\ell=-1-s$, and equating their values on $A_{01}+A_{11}+A_{21}$ gives $rn=0$, so $r=0$.  So this partition gives a fusion in the imprimitive case where $\ell=-1-s$ and $r=0$.

\vspace{0.2cm}

$\uptau' = 2456|3789$: The values of $\chi_{11}$, $\chi_{21}$, and $\chi_{10}$ on $A_{10}+A_{01}+A_{11}+A_{21}$ satisfy $k+rn > r > s$, and hence we have at least four distinct rows $\chi_{00}, ~ \chi_{01}, ~ \chi_{11}$ and $\chi_{21}$ in the modified character table.  By the Bannai-Muzychuk criterion, this partition will not produce a fusion of rank $3$. 

\vspace{0.2cm}

$\uptau' = 2789|3456$: This time we see that the values of $\chi_{20}$, $\chi_{11}$, and $\chi_{21}$ on $A_{10}+A_{02}+A_{12}+A_{22}$ satisfy $k+n(-1-s) > r > s$, so as shown above we have at least four distinct rows in the modified character table. Hence,this partition will not result in a rank $3$ fusion. 

\medskip
This completes the characterization of fusions of $(\mathcal{A} \otimes 1) \wr (1 \otimes \mathcal{A})$.  For the fusions of $(1 \otimes \mathcal{A}) \wr (\mathcal{A} \otimes 1)$, we need only apply the permutation $(24)(37)(68)$ to each partition giving a fusion of $(\mathcal{A} \otimes 1) \wr (1 \otimes \mathcal{A})$.  
\end{proof}

\section{Classifying fusions of the tensor square}
\label{sec:classifying}

In the previous section, we classified all possible fusions of the wreath products of symmetric rank 3 association schemes $\mathcal{A}$ with themselves.  In this section, aided by a computer program, we will apply the Bannai-Muzychuk criterion to produce guaranteed and special case fusions of the tensor square $\mathcal{A} \otimes \mathcal{A}$, and in the next section we will explain why these are all the possible fusions. 

We continue with the notation of \S 2, so $\mathcal{A} = \{A_0,A_1,A_2\}$, where $A_1$ is the adjacency matrix of a strongly-regular graph $\Gamma$ of order $n$ with eigenvalues $k \ge r \ge 0 > -1 \ge s$, and $A_2$ is the adjacency matrix of its complement $\bar{\Gamma}$, with valency $\ell = n - k - 1$ and eigenvalues $\ell$, $-1-s$, and $-1-r$.  Let $\chi_0$ be the valency character, and $\chi_1$ and $\chi_2$ be the other irreducible characters of the adjacency algebra, with $\chi_1(A_1)=r$ and $\chi_2(A_1)=s$, so the rows of the character table of $\mathbb{C}\mathcal{A}$ correspond to $\chi_0$, $\chi_1$, and $\chi_2$.  (We remark that the calculations in this section do not depend on $A_1$ being the adjacency matrix of a graph, we get the same fusions if $\mathcal{A}$ is only assumed to be the basis of a rank $3$ symmetric standard integral table algebra.)  With our assumptions on eigenvalues, the character table of $\mathcal{A} \otimes \mathcal{A}$, with columns indexed by $A_{ij}$ and columns indexed by $\chi_{ij} := \chi_i \otimes \chi_j$, is the following - with the $A_{00}$ column of all $1$'s omitted:  

{\tiny
\[
    \begin{blockarray}{cccccccccc}
          & A_{10} & A_{20} & A_{01} & A_{11} & A_{21} & A_{02} & A_{12} & A_{22}  \\
      \begin{block}{c[cccccccc]c}
     \chi_{00}     &    k      &     \ell  &    k      &    k^2  &   \ell k  &    \ell    &   k\ell  &   \ell^2 &  1 \\
\chi_{01}     &    k      &     \ell  &    r      &     kr    &   \ell r   & (-1-r)   &  k(-1-r) &  \ell(-1-r) & m_r \\
\chi_{02}    &    k      &     \ell  &    s      & k(-1-r)  &  \ell s   & (-1-s)  &  k(-1-s) &  \ell(-1-s) & m_s \\
\chi_{10}     &    r      &  (-1-r)  &    k      &  rk       & (-1-r)k  &    \ell   &  r \ell    &  (-1-r)\ell & m_r  \\
\chi_{11}      &    r      &  (-1-r)  &    r      &  r^2     & (-1-r)r  &  (-1-r)  &  r(-1-r)  &  (-1-r)^2 & m_r^2 \\
\chi_{12}     &    r      &  (-1-r)  &    s     &   r(-1-r) & (-1-r)s  &  (-1-s) &  r(-1-s)  &  (-1-r)(-1-s) & m_rm_s \\
\chi_{20}     &    s     &  (-1-s)  &    k     &   sk       & (-1-s)k &   \ell    &  s \ell    &  (-1-s) \ell   &  m_s \\
\chi_{21}       &    s     &  (-1-s)  &    r     &   sr        & (-1-s)r &   (-1-r) &  s(-1-r)  &  (-1-s)(-1-r) & m_rm_s \\
\chi_{22}     &    s     &  (-1-s)  &    s     &   s^2 & (-1-s)(-1-r) & (-1-s) & s(-1-s) &  (-1-s)^2 & m_s^2 \\
      \end{block}
    \end{blockarray}
\]}

As done in \S \ref{notation}, we relabel the $A_{ij}$ as $C_{3j+i+1}$ so that the columns of the character table are labelled with the indices $1,\dots,9$.  Then each fusion of $\mathcal{A} \otimes \mathcal{A}$ is naturally associated with a partition of $\{2,\dots,9\}$.  Using a computer, we check the condition of Lemma \ref{Bannai-Muzychuk} to find all of the partitions that will give a fusion of $\mathcal{A} \otimes \mathcal{A}$ for all symmetric rank $3$ table algebras $\mathcal{A}$.  The list of partitions corresponding to these guaranteed fusions of $\mathcal{A}\otimes \mathcal{A}$ matches the list given in \cite[\S 6]{HermanJoshiMeagher2022}.

\begin{Lemma} \label{Lemma:GuaranteedLattice}
The $13$ partitions of $\{2,\dots,9\}$ corresponding to the guaranteed fusions of $\mathcal{A} \otimes \mathcal{A}$ are: 
 
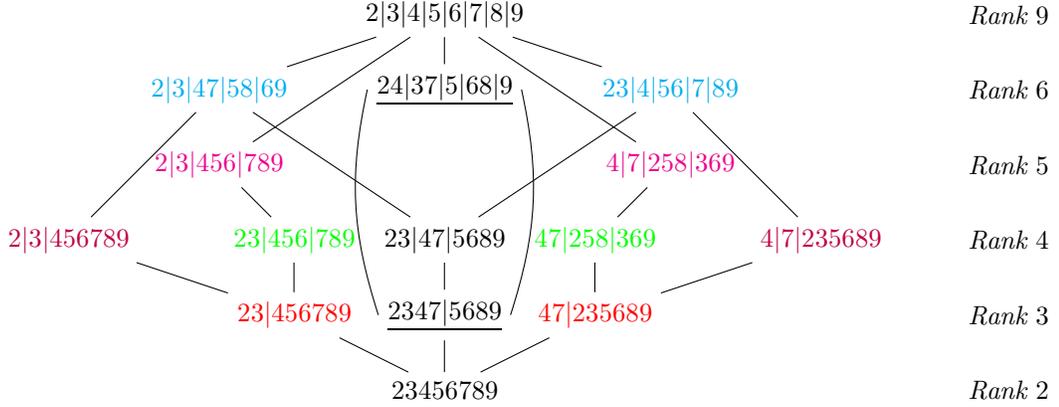
\begin{center}
         \begin{tikzpicture}[scale=.5]
  \node (one) at (0,8) {$2|3|4|5|6|7|8|9$};
  \node (label) at (15,8) {{Rank $9$}};
    \node (d) at (-6,6) {$ \textcolor{cyan}{2|3|47|58|69}$};
  \node (c) at (0,6) {\underline{${24|37|5|68|9}$}};
  \node (1) at (6,6) {$\textcolor{cyan}{23|4|56|7|89}$}  ;

  \node (label) at (15,6) {{Rank $6$}};
  \node (b) at (-6,4) {$\textcolor{magenta}{2|3|456|789}$};
  \node (f) at (6,4) {$\textcolor{magenta}{4|7|258|369}$};

  \node (label) at (15,4) {{Rank $5$}};
  \node (p) at (-4,2) {$\textcolor{green}{23|456|789}$};
  \node (q) at (4,2) {$\textcolor{green}{47|258|369}$};

  \node (zero) at (-10,2) {$\textcolor{purple}{2|3|456789}$};
  \node (zeru) at (0,2) {$23|47|5689$};
  \node (zeroes) at (10,2) {$\textcolor{purple}{ 4|7|235689}$};
  \node (label) at (15,2) {{Rank $4$ }};
  \node (zeri) at (0,0) {\underline{$2347|5689$}};
  \node (rank3) at (-4,0) {\textcolor{red}{$23|456789$}};
  \node (Rank3) at (4,0) {\textcolor{red}{$47|235689$}};
  \node (label) at (15,0) {{Rank $3$}};
  \node (zery) at (0,-2) {$23456789$};
  \node (label) at (15,-2) {{Rank $2$}};
  
  \draw (one) -- (d) (one) -- (c) (one) -- (1)
(d) -- (zero) --  (rank3) -- (p)  (rank3) -- (zery) (1) -- (zeroes) -- (Rank3) -- (q) (Rank3) -- (zery) (zeru) -- (zeri)  -- (zery) (one) -- (b) -- (p)  (one) -- (f) -- (q) (d) -- (zeru) (1) -- (zeru)
  ;
 \draw  (zeri.east) to[bend right=15] (c.east);
 \draw  (c.west)to[bend right=15] (zeri.west);
\end{tikzpicture}
    \end{center}

%$$ 2|3|4|5|6|7|8|9, $$ 
%$$ 24|37|5|68|9, 23|4|56|7|89 \leftrightarrow 2|3|47|58|69, $$ 
%$$ 2|3|456|789 \leftrightarrow 4|7|258|369, $$
%$$ 23|456|789 \leftrightarrow 47|258|369, $$
%$$ 2|3|456789 \leftrightarrow 4|7|235689, 23|47|5689, $$
%$$ 2347|5689, $$ 
%$$ 23456789.$$
(Here highlighted fusions in the same color at a level indicates a pair of dual fusions obtained from applying the flip map $A_{ij} \mapsto A_{ji}$ to the tensor product, which is obtained when the permutation $(24)(37)(68)$ is applied to the partition. The underline partitions give the generalized Hamming scheme and its guaranteed fusion.) 
\end{Lemma}

We remark that, in addition to the duality obtained from the flip map $A_{ij} \leftrightarrow A_{ji}$, each partition corresponding to a fusion of $\mathcal{A} \otimes \mathcal{A}$ above has a {\it switch partner} that corresponds to switching the order of $A_1$ and $A_2$ in both copies of $\mathcal{A}$.  The switch partner is obtained by applying the permutation $(24)(37)(59)(68)$ to the partition.  The above partitions are all fixed by this operation.   

\medskip
Next, for the other partitions of $\{2,\dots,9\}$, after we sum the columns of the character table according to the partition, equality of any pair of rows imposes conditions on our parameters.  For the partition to correspond to a fusion, the number of distinct rows that result has to match the size of the partition.  From our earlier results in \cite{HermanJoshiMeagher2022}, we know some families of $\mathcal{A}$'s where $H(2,\mathcal{A})$ has a special case fusion, so these will also have special case fusions for their tensor square.  It is easier to apply the Bannai-Muzychuk criterion to find the partitions corresponding to the fusions of their tensor squares directly, so we will do that first.   

\subsection{Strongly-regular graph families admitting special case fusions}

The families of strongly-regular graphs whose generalized Hamming scheme has special case fusions were found in \cite{HermanJoshiMeagher2022}.  Here we consider the fusions of the tensor square for these families of graphs. The full lists of nontrivial fusions corresponding to each of these families is also available in the {\tt TensorProductFusions.txt} file in our GitHub repository \cite{GitHubRepository}. 

\medskip
The first family corresponds to the imprimitive strongly-regular graphs, for which either $k=r$ and $s=-1$, or $\ell =-1-s$ and $r=0$. 

\begin{theorem}\label{Completegraphs}
If $\Gamma$ is a union of complete graphs and $\bar{\Gamma}$ is a complete multipartite graph, then the character table for the association scheme is

\[
\mathcal{P}(\mathcal{A})= 
    \begin{blockarray}{cccc}
      \begin{block}{[ccc]c}
        1 & r & m(1+r) & 1\\
       1 & r & -1-r  & m\\
    1 & -1 & 0  & r(1+m)\\
      \end{block}
    \end{blockarray}.
\]

For the above character table, we get the following $45$ additional fusions other than the guaranteed fusions in Lemma \ref{Lemma:GuaranteedLattice}. The partitions in red correspond to special case fusions of the generalized Hamming scheme. The partitions in the same color at a level indicates a pair of dual fusions obtained from applying the flip map $A_{ij} \mapsto A_{ji}$ to the tensor product, which is obtained when the permutation $(24)(37)(68)$ is applied to the partition.
\end{theorem}

{\footnotesize
 \begin{center}
         \begin{tikzpicture}[scale=.5]
  \node (rank8) at (-5,22) {\textcolor{green}{${2|3|4|5|6|78|9}$}};
  \node (rank8) at (5,22) {\textcolor{green}{$2|36|4|5|7|8|9$}};
  \node (label) at (15,22) {{\it rank $8$}};
  %%%%%%%%%%%%%%%%%%%%%%%%%%%%%%%%%%%%%%
  \node (rank7) at (-10,19) {\textcolor{green}{${2|3|4|5|6|789}$}};
  \node (rank7) at (-5,19) {\textcolor{blue}{${2|3|45|6|78|9}$}};
  \node (rank7) at (0,19) {${2|36|4|5|78|9}$};
  \node (rank7) at (10,19) {\textcolor{green}{${2|369|4|5|7|8}$}};
  \node (rank7) at (5,19) {\textcolor{blue}{${25|36|4|7|8|9}$}};
  \node (label) at (15,19) {{\it rank $7$}};

\node (label) at (15,16) {{\it rank $6$}};
  
  \node (rank6) at (-10,16) {\textcolor{blue}{${2|3|45|6|789}$}};
  \node (rank6) at (-5,16) {\textcolor{green}{${2|36|45|78|9}$}};
  \node (rank6) at (0,16) {${2|3678|4|5|9}$};
  \node (rank6) at (5,15) {\textcolor{olive}{${2|369|4|5|78}$}};
  \node (rank6) at (0,15) {${24|36|5|78|9}$};
  \node (rank6) at (-5,15) {\textcolor{olive}{${2|36|4|5|789}$}};
  \node (rank6) at (5,16) {\textcolor{green}{${25|36|4|78|9}$}};
  \node (rank6) at (10,16) {\textcolor{blue}{${25|369|4|7|8}$}};
 
\node (label) at (15,12) {{\it rank $5$}};

  \node (rank5) at (-10,12) {\textcolor{blue}{${2|36|45|789}$}};
   \node (rank5) at (10,12) {$\textcolor{blue}{25|369|4|78}$};
  \node (rank5) at (-5,12) {\textcolor{green}{${2|3|456|789}$}};
   \node (rank5) at (0,12) {${2|36789|4|5}$};
    \node (rank5) at (5,12) {\textcolor{green}{${258|369|4|7}$}};
     \node (rank5) at (-5,11) {\textcolor{cyan}{${2|3|4578|69}$}};
      \node (rank5) at (-10,11) {\textcolor{magenta}{${23|4|56|789}$}};

  \node (rank5) at (0,11) {${245|36|78|9}$};
   \node (rank5) at (5,11) {\textcolor{cyan}{${2356|4|7|89}$}};
    \node (rank5) at (10,11) {\textcolor{magenta}{${2|369|47|58}$}};
     \node (rank5) at (-10,10) {\textcolor{violet}{${2|3678|45|9}$}};
      \node (rank5) at (-5,10) {\textcolor{lime}{${2|369|45|78}$}};
       \node (rank5) at (0,10) {\textcolor{red}{${24|3678|5|9}$}};
        \node (rank5) at (5,10) {\textcolor{lime}{${25|36|4|789}$}};
         \node (rank5) at (10,10) {\textcolor{violet}{${25|36|4|789}$}};
     \node (rank5) at (-5,9) {\textcolor{orange}{$24|36|5|789$}};
     \node (rank5) at (5,9) {\textcolor{orange}{$24|369|5|78$}};

     \node (label) at (15,6) {{\it rank $4$}};

  \node (rank4) at (-10,6) {\textcolor{green}{${2|3456|789}$}};
  \node (rank4) at (-5,6) {\textcolor{violet}{${2|36789|45}$}};
   \node (rank4) at (0,6) {\textcolor{red}{${24|36789|5}$}};
    \node (rank4) at (10,6) {\textcolor{green}{${2578|4|369}$}};
     \node (rank4) at (-10,5) {\textcolor{cyan}{${2356|4|789}$}};
      \node (rank4) at (5,5) {\textcolor{purple}{${245|36|789}$}};
 \node (rank4) at (-5,5) {\textcolor{purple}{${245|369|78}$}};
  \node (rank4) at (5,6) {\textcolor{violet}{${25|36789|4}$}};
   \node (rank4) at (0,5) {\textcolor{red}{${245|3678|9}$}};
    \node (rank4) at (10,5) {\textcolor{cyan}{${4578|369|2}$}};
  \draw ;

\node (label) at (15,2) {{\it rank $3$}};
 \node (rank3) at (-10,2) {\textcolor{magenta}{${2|3456789}$}};
 \node (rank3) at (-5,2) {\textcolor{lime}{${23456|789}$}};
 \node (rank3) at (10,2) {\textcolor{magenta}{${2356789|4}$}};
 \node (rank3) at (5,2) {\textcolor{lime}{${24578|369}$}};
 \node (rank3) at (0,2) {\textcolor{red}{${245|36789}$}};
  
\end{tikzpicture}
    \end{center}
}

If $\Gamma$ is the complete multipartite graph, so $A_1$ and $A_2$ switch roles in the imprimitive association scheme, then we will have $r=0$ and $k=-s$, and the list of proper non-trivial fusions will be the switch partners of the fusions listed above.  These are obtained by applying the permutation $(23)(47)(59)(68)$ to the above partitions.  

\medskip
The next family is the family of conference graphs, which includes the symmetric Payley graphs and all their cospectral strongly-regular graphs. 

\begin{theorem}\label{Payley}
If $k=\ell$ and $s=-1-r$ then the character table of $\mathcal{A}$ has the form 

\[
\mathcal{P}(\mathcal{A})= 
    \begin{blockarray}{cccc}
      \begin{block}{[ccc]c}
        1 & 2(r+r^{2}) & 2(r+r^{2}) \bigstrut[t] & 1\\
1 & r & -1-r & 2(r+r^{2})\\
1 & -1-r & r & 2(r+r^{2})\\
      \end{block}
    \end{blockarray}.
\]

For the above character table, we get the following $11$ additional fusions other than the guaranteed fusions in Lemma \ref{Lemma:GuaranteedLattice}.

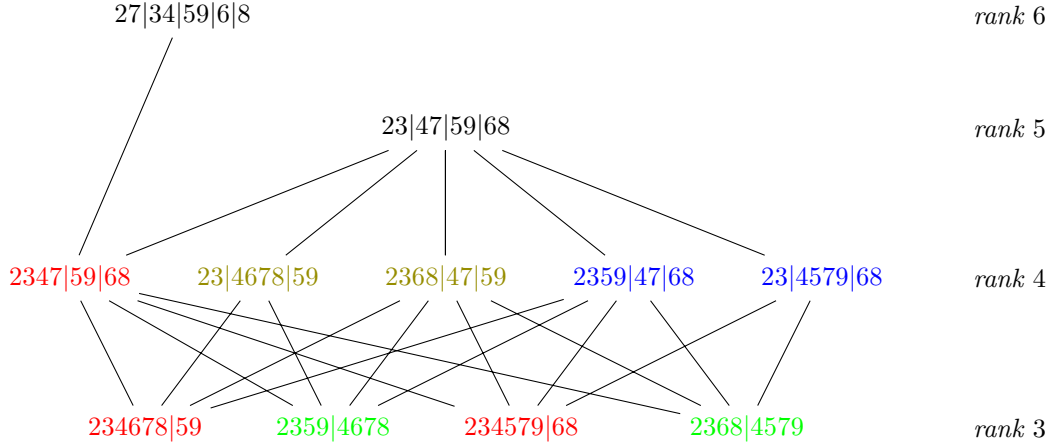
 \begin{center}
         \begin{tikzpicture}[scale=.5]
  \node (rank6) at (-7,9) {$27|34|59|6|8$};
  \node (label) at (15,9) {{\it rank $6$}};
  \node (rank5) at (0,6) {${23|47|59|68}$};

  \node (label) at (15,6) {{\it rank $5$}};
  \node (rank4i) at (10,2) {\textcolor{blue}{${23|4579|68}$}};
  \node (rk4ii) at (5,2) {\textcolor{blue}{${2359|47|68}$}};

  \node (rk4iii) at (-5,2) {\textcolor{olive}{${23|4678|59}$}};
  \node (rk4iv) at (-10,2) {\textcolor{red}{$2347|59|68$}};
  \node (rk4v) at (0,2) {\textcolor{olive}{${ 2368|47|59}$}};
  \node (label) at (15,2) {{\it rank $4$ }};
  \node (rank3i) at (-8,-2) {\textcolor{red}{$234678|59$}};
  \node (Rank3ii) at (2,-2) {\textcolor{red}{$234579|68$}};
  \node (rank3iii) at (-3,-2) {\textcolor{green}{$2359|4678$}};
  \node (Rank3iv) at (8,-2) {\textcolor{green}{$2368|4579$}};
  \node (label) at (15,-2) {{\it rank $3$}};

  \draw (rank6) -- (rk4iv) -- (rank5) -- (rk4iii) (rank5) -- (rank4i) (rank5) -- (rk4v) --  (rank5) -- (rk4ii) (rank3iii) -- (rk4iv) -- (rank3i) (Rank3ii) -- (rk4iv) -- (Rank3iv) (rank3iii) -- (rk4iii) -- (rank3i)
  (Rank3iv) -- (rank4i) -- (Rank3ii) (Rank3iv) -- (rk4v) -- (Rank3ii) (rank3iii) -- (rk4v) -- (rank3i) (Rank3iv) -- (rk4ii) -- (Rank3ii) (rank3iii) -- (rk4ii) -- (rank3i)
  ;
  \draw ;
\end{tikzpicture}
    \end{center}
(Here highlighted fusions in red indicate the self-dual partitions that give fusions of the generalized Hamming scheme. The fusions in the same color at a level indicates a pair of dual fusions obtained from applying the flip map $A_{ij} \mapsto A_{ji}$ to the tensor product, which is obtained when the permutation $(24)(37)(68)$ is applied to the partition.) 

\end{theorem}

\begin{proof}
We applied the orthogonality conditions from Section $2$ to the conditions $k=\ell$ and $s=-1-r$ and obtained $k=\ell=2(r+r^2)$. Hence, our character table can be expressed in one variable. These fusions were checked using Lemma \ref{Bannai-Muzychuk} using GAP \cite{GAP4}.  For the GAP code see {\tt TensorProductFusions.txt} file on our GitHub respository \cite{GitHubRepository}. 
\end{proof}

\begin{theorem}\label{News}
If $k=s^2$, $\ell=-2s$ with $r=1$, then the character table for the association scheme is

\[
\mathcal{P}(\mathcal{A})= 
    \begin{blockarray}{cccc}
      \begin{block}{[ccc]c}
        1 & s^2 & -2s & 1\\
       1 & 1 & -2  & s^2\\
    1 & s & -1-s  & -2s\\
      \end{block}
    \end{blockarray}.
\]

\begin{enumerate}[label=(\roman*)]

\item If $\Gamma$ is the cartesian product of two complete graphs with order $s+1$ implying $k=s^2$ and $r=1$, then for the above character table, we get the following rank $5$ fusion other than the guaranteed fusions in Lemma \ref{Lemma:GuaranteedLattice}.
$$ 249|37|5|68 $$

\item If $\Gamma$ is the complement of the cartesian product of two complete graphs with order $r+2$ implying $k=2(1+r)$, $\ell=(1+r)^{2}$ and $s=-2$, then there is an additional rank $5$ fusion of $\mathcal{A} \otimes \mathcal{A})$ which is the switch partner of part one above:
$$ 24|375|68|9 $$
\end{enumerate}

\end{theorem}

\begin{theorem}\label{Crazyrank4}
If $k=3-s-r$ and $\ell=5+s+r$ then the character table of the association scheme, $\mathcal{A}$ has the form 

\[
\mathcal{P}(\mathcal{A})= 
    \begin{blockarray}{cccc}
      \begin{block}{[ccc]c}
       1 & 3-s-r & 5+s+r  & 1\\
       1 & r & -1-r  & f\\
    1 & s & -1-s & g \\
      \end{block}
    \end{blockarray}.
\]

For the above character table, we get the following additional fusion other than the guaranteed fusions in Lemma \ref{Lemma:GuaranteedLattice}.
$$ 249|37|5|68 $$
\end{theorem}

It was noted in \cite{HermanJoshiMeagher2022} that only three strongly-regular graphs satisfy the conditions of Theorem \ref{Crazyrank4} are the ones of order $9$: the $9$-gon of valency $2$, its complement of valency $6$, and the the Payley graph of order $9$ (with its isomorphic complement).  

\begin{theorem}\label{Clebsch}
If $k=r(3+r)$ and $\ell=(3+r)$ with $s=-2$ then the character table, $\mathcal{P}(\mathcal{A})$ has the form 
\[
\mathcal{P}(\mathcal{A})= 
    \begin{blockarray}{cccc}
      \begin{block}{[ccc]c}
       1 & r(3+r) & 3+r  & 1\\
       1 & r & -1-r & (3+r)\\
       1 & -2 & 1 & r(3+r)\\
      \end{block}
    \end{blockarray}.
\]

\begin{enumerate}[label=(\roman*)]

\item For the above character table, we get the following additional fusion other than the guaranteed fusions in Lemma \ref{Lemma:GuaranteedLattice}.
$$ 249|35678 $$

\item The above fusion has a switch partner which gives the rank $3$ fusion, 
$$ 24689|357 $$
The graph $\Gamma$ in this case satisfies the conditions, $\ell=-k(s+1)$ and $k=2-s$ with $r=1$.
\end{enumerate}

\end{theorem}

\begin{theorem}\label{Clebsch2}
If $k=r(2r+1)$, $\ell=(r-1)(2r+1)$, and $s=-r$, then the character table $\mathcal{P}(\mathcal{A})$ has the form 

\[
\mathcal{P}(\mathcal{A})= 
    \begin{blockarray}{cccc}
      \begin{block}{[ccc]c}
        1 & r(2r+1) & (r-1)(2r+1) \bigstrut[t] & 1\\
        1 & r & -1-r & r(2r+1)\\
        1 & -r & -1+r & (r-1)(2r+1)\\
      \end{block}
    \end{blockarray}.
\]

\begin{enumerate}[label=(\roman*)]
\item For the above association scheme, the tensor square has the following rank $3$ fusion other than the guaranteed fusions in Lemma \ref{Lemma:GuaranteedLattice}.
$$ 2468|3579 $$

\item The above fusion has a switch partner which gives the rank $3$ fusion,  
$$ 2459|3768 $$
This occurs in the case $k=(s+2)(2s+1)$, $\ell=(s+1)(2s+1)$, and $r=-2-s$.

\end{enumerate}
\end{theorem}

As pointed out in \cite{HermanJoshiMeagher2022}, $r$ must be $\ge 2$ in the first part of Theorem \ref{Clebsch2}.  The character tables in Theorem \ref{Clebsch} and \ref{Clebsch2} coincide in the case $r=2$, in this case the association scheme is that of the Clebsch graph.  Strongly-regular graphs with this character table are known to exist for $2 \le r \le 6$, beyond that they are not always known to exist.

\section{Our main result}

Let $\mathcal{A} = \{A_0 = I, A_1, A_2 \}$ be the basis of adjacency matrices of a symmetric rank $3$ association scheme, or the standard basis of a rank $3$ symmetric table algebra.  In this section, we will describe how we have used a computer program to verify that the fusions listed in Theorems \ref{Completegraphs}--\ref{Clebsch2} are all of the possible fusions of $\mathcal{A} \otimes \mathcal{A}$.  

\begin{theorem}\label{mainresult}
If $\mathcal{A}$ does not belong to the families of strongly-regular graphs listed in Theorems \ref{Completegraphs}--\ref{Clebsch2}, then the only fusions of the tensor square of $\mathcal{A}$ are the $13$ guaranteed fusions listed in Lemma \ref{Lemma:GuaranteedLattice}.

\end{theorem}

\begin{proof}
We know that the rank of a fusion corresponds to the number of columns in the character table which are analysed case by case and are referenced by their category in the file {\tt AllPartitionsDataCategorized.txt} on our GitHub repository \cite{GitHubRepository}. The special case fusions covered in Theorems \ref{Completegraphs}--\ref{Clebsch2} are also verified there using GAP. 

\medskip
All that is left to show is that the only nontrivial special case fusions occur when $\mathcal{A}$ belongs to one of the families in Theorems \ref{Completegraphs}--\ref{Clebsch2}.  The first $45$ non-trivial fusions (these are the fusions listed in Lines 129-199 of {\tt TensorProductFusions.txt} in \cite{GitHubRepository}) hold only in the rank $3$ imprimitive case and hence are ruled out for any other strongly-regular graphs or table algebras in our search if we assume that $s<-1$ and $k > r$. Also, their switch partners occur only when $r=0$ and $\ell=-1-s$, so we can assume $r>0$ and $\ell > -1-s$.

Focusing on the primitive case, we define a set $U$ (given in lines 14-29 of  {\tt TensorProductFusions.txt} in \cite{GitHubRepository}) which consists of polynomials that do not satisfy the parameters of a primitive strongly-regular graph. This set is used to sieve out partitions which require one of these polynomial conditions to give a fusion. If an equality of a pair of irreducible characters of the fusion induced by a partition requires a polynomial in $U$ to equal $0$, this pair of irreducible characters cannot hold in the fusion.  By the Bannai-Muzychuk criterion in Lemma \ref{Bannai-Muzychuk}, the partition will only induce an actual fusion when the partition it induces on irreducible characters has the same size.  Our step-by-step sieving process is explained in {\tt ReadMe.txt} file on \cite{GitHubRepository}.  

We will give four typical examples here, two examples of partitions that produce fusions and two examples of partitions that do not produce fusions. All of the other partitions of $\{2,\dots,9\}$ are dealt with in a similar fashion in \cite{GitHubRepository}.  

\medskip
\begin{enumerate}

\item Consider the partition $ 27|34|59|6|8$ given in Lines 263-267 of the file \\ {\tt AllPartitionsDataCategorized.txt } in \cite{GitHubRepository} that gives a rank $6$ fusion.
For this partition to give a nontrivial fusion, the row equalities, $\chi_{2}=\chi_{7}$, $\chi_{3}=\chi_{4}$, and $\chi_{5}=\chi_{9}$ must be satisfied. On further investigation, we can see that $\chi_{5}=\chi_{9}$ implies $r=-1-s$, and $\chi_{2}=\chi_{7}$, $\chi_{3}=\chi_{4}$ implies $k=\ell=2r+2r^2=2s+2s^2$. Hence, it is proved that this rank $6$ nontrivial fusion falls under the strongly-regular graph family given in Theorem \ref{Payley}.

\medskip
\item Consider the partition $ 249 |35678 $ given in Lines 391-413 of the above file that gives a rank $3$ fusion.
For this partition to give a nontrivial fusion, the row equalities, $\chi_{2}=\chi_{6}=\chi_{8}=\chi_{9}$, and $\chi_{3}=\chi_{4}==\chi_{5}=\chi_{7}$ must be satisfied. On further investigation, we can see that the first set of row equalities implies $r=1$, $s=-2$, and the later set of row equalities  implies $k= r\ell$, with $\ell=3+r$. Hence, it is proved that this rank $3$ nontrivial fusion falls under the strongly-regular graph family given in Theorem \ref{Clebsch}.

\medskip
\item Consider the partition $ 23489 |567 $ given in Lines 521-534 of the same file. Although initial investigation shows that this partition satisfies the Bannai-Muzychuk criterion and gives a rank $3$ fusion with row equalities, $\chi_{2}=\chi_{3}=\chi_{5}=\chi_{6}=\chi_{7}$, and $\chi_{4}=\chi_{8}=\chi_{9}$. On further investigation, we can see that the first set of row equalities implies $kr+\ell r-r+2s=0$ with $k-\ell=1-2r$, and the later set of row equalities  implies $ks+\ell s+2r-s=0$. Since, all of these conditions must be satisfied, we next check for compatibility of the row equalities. We combine the following two equations $$ 
kr+\ell r-r+2s=0, ~ \textrm{and} ~ks+\ell s+2r-s=0,
$$ to get $k+\ell =3$. Since we also have $k-\ell=1+2r$ this implies $\ell=1-r$ which is a contradiction. Hence, we prove that that this rank $3$ partition does not give a nontrivial fusion.

\medskip
\item Consider the partition $2678|34|59$ given in Lines 2118-2124 of the same file in (1). A quick glance on the row equality conditions doesn't bring out any obvious contradictions. In this case we can observe that even if all the row equalities are satisfied and compatible with each other, this implies $\chi_{2}=\chi_{6}=\chi_{7}$, $\chi_{3}=\chi_{4}$ and $\chi_{5}=\chi_{9}$ we still end up with at least $5$ distinct rows in the modified character table. Hence, using Lemma \ref{Bannai-Muzychuk} we prove that that this rank $4$ partition does not give a nontrivial fusion.

\end{enumerate}

\medskip
After running the sieve, we find that all of the remaining partitions corresponding to special case fusions are either the ones listed for one of the $6$ strongly-regular graph families in the previous section, or their switch partners.  So this verifies the theorem. 
\end{proof}

%\printbibliography %Prints bibliography
\bibliographystyle{plain}
\bibliography{sample}

\end{document}